\documentclass{article}

\usepackage{amsmath,amssymb,amsthm,enumerate}
\usepackage{comment,color,graphicx,subfigure,tikz,tikz-cd}

\newcommand\N{\mathbb N}
\newcommand\Z{\mathbb Z}
\newcommand\R{\mathbb R}
\newcommand\C{\mathbb C}
\newcommand{\Aut}{\mathrm{Aut}}
\newcommand{\End}{\mathrm{End}}
\DeclareMathOperator{\Sym}{Sym}
\DeclareMathOperator{\GL}{GL}

\DeclareMathOperator{\Homeo}{Homeo}
\DeclareMathOperator{\id}{id}

\theoremstyle{plain}
\newtheorem{lemma}{Lemma}
\newtheorem{theorem}{Theorem}
\newtheorem{metalemma}{Metalemma}
\newtheorem{proposition}{Proposition}
\newtheorem{corollary}{Corollary}

\newtheorem{remark}{Remark}

\theoremstyle{definition}
\newtheorem{definition}{Definition}

\newcommand{\Set}{{\normalfont\textbf{Set}}}
\newcommand{\CountSet}{{\normalfont\textbf{CountSet}}}
\newcommand{\FinSet}{{\normalfont\textbf{FinSet}}}
\newcommand{\FinVect}{{\normalfont\textbf{FinVect}}}
\newcommand{\Top}{{\normalfont\textbf{Top}}}
\newcommand{\Pos}{{\normalfont\textbf{Pos}}}

\newcommand{\proj}{\mathrm{proj}}

\title{Alternation diameter of a product object}

\author{Ville Salo}

\begin{document}
\maketitle

\begin{abstract}
We prove that every permutation of a Cartesian product of two finite sets can be written as a composition of three permutations, the first of which only modifies the left projection, the second only the right projection, and the third again only the left projection, and three alternations is indeed the optimal number. We show that for two countably infinite sets, the corresponding optimal number of alternations, called the alternation diameter, is four. The notion of alternation diameter can be defined in any category. In the category of finite-dimensional vector spaces, the diameter is also three. For the category of topological spaces, we exhibit a single self-homeomorphism of the plane which is not generated by finitely many alternations of homeomorphisms that only change one coordinate. The results on finite sets and vector spaces were previously known in the context of memoryless computation.
\end{abstract}

\section{Introduction}

We prove in this paper that every permutation of a Cartesian product of two finite sets can be performed by first permuting on the left, then on the right, then on the left. This is Theorem~\ref{thm:Two}, and its proof is a reduction to Hall's marriage theorem. For two countably infinite sets, every permutation of the Cartesian product can be performed either by permuting left-right-left-right or right-left-right-left (sometimes only one of these orders works). This is Theorem~\ref{thm:CountableCase}. Its proof is elementary set theory, a Hilbert's Hotel type argument. We also prove that for the direct product of two finite-dimensional vector spaces, every linear automorphism can be written by a composition of three linear automorphisms, the first of which only modifies the left coordinate, the second the right coordinate and the third the left coordinate. This is Theorem~\ref{thm:FinVect}. In Theorem~\ref{thm:Homeo}, we show that a similar result does not hold for general self-homeomorphisms of the plane.

It turns out that the results on finite sets and vector spaces have been independently proved several times. See Section~\ref{sec:Existing} for this and other related results. 

The study of such alternations arises from the unpublished draft \cite{SaSc16a} where this optimal number of alternations was studied for automorphism groups of subshifts. Here, we note that this notion applies to the automorphism group of a product object in any category, and study it in some categories of interest. We call the optimal length of a left-right alternation the \emph{alternation diameter} (see below for more detailed definitions).

The alternation diameter is interesting as a general concept since the automorphism group of a product object always contains the automorphism groups of the left and right components (in a natural way) -- and more generally what we call the left and right groups (defined below) --, but in addition it can contain many ``entirely new'' automorphisms that can only be understood globally, and do not easily reduce to the study of the left or right component of the product separately. In the case that the alternation diameter is bounded for a particular product object and the left and right groups are easy to describe, we get a handle on the elements of the automorphism group, at least as a set. Of course, in complicated categories we cannot expect this to happen very generally, but it can be a helpful technique in the study of automorphism groups of individual objects when it succeeds.

In this note we consider some simple categories where alternation diameter is actually globally bounded over the whole category (though in these simple cases it does not really help in understanding the automorphism group of a product), and show by examples that the left and right groups, not surprisingly, do not in general generate the automorphism group of a product object in more complex categories.

In terms of the alternation diameter, our results are the following: 

\begin{theorem}
The category
\begin{itemize}
\item $\FinSet$ of finite sets and functions has alternation diameter $3$,
\item $\CountSet$ of countable sets and functions has alternation diameter $4$,
\item $\FinVect$ of finite-dimensional vector spaces and linear maps has alternation diameter $3$,
\end{itemize}
\end{theorem}

Our examples of categories where alternation diameter is undefined (meaning that left and right alternations do not generate all automorphisms) are the following: The plane $\R \times \R$ has, for slightly non-trivial reasons, undefined alternation diameter in the category $\Top$ of topological spaces and continuous functions. The square $[0,1] \times [0,1]$ has undefined alternation diameter in $\Top$ for trivial reasons. In $\Pos$, the product posets $[n] \times [n]$ and $D \times D$ where $D$ is the \emph{diamond}, i.e. the poset with Hasse diagram $\diamond$, have undefined alternation diameter for (different) trivial reasons. (By a more \Pos-specific proof, Maximilien Gadouleau has shown that in fact $X \times X$ has undefined alternation diameter for every finite poset $X$.)

We now give some more detailed definitions (see \cite{Ma71} for basic examples and notions of category theory\footnote{Other than Section~\ref{sec:GeneralDef} and Section~\ref{sec:Posets}, we are not concerned with the ``theory'', mainly terminology and concepts.}). In the category of sets and functions a product object $A \times B$ is just the Cartesian product of $A$ and $B$, and the automorphism group of an object is just the full permutation group on that set. In concrete categories where products behave this way, for a product object $A \times B$ and its automorphism group $G$, we define the \emph{Left group} $G_L$, namely the subgroup of $G$ containing those automorphisms of $A \times B$ that modify only the left coordinate, i.e. satisfy $\forall a \in A, b \in B: \exists c \in A: g(a, b) = (c, b)$. Define the \emph{Right group} $G_R$ analogously. Our precise statements about sets are that in the category $\FinSet$ of finite sets, $G = G_LG_RG_L = G_RG_LG_R$, and in the category $\CountSet$ of countable sets, $G = G_LG_RG_LG_R \cup G_RG_LG_RG_L$ but generally $G \notin \{G_LG_RG_LG_R, G_RG_LG_RG_L\}$).

In less $\Set$-like categories one can define $G_L$ and $G_R$ in terms of commutative diagrams (see Section~\ref{sec:GeneralDef} for the diagrams): let $A \times B$ be the category-theoretic product of $A$ and $B$, and $\proj_A : A \times B \to A$ and $\proj_B : A \times B \to B$ the projections defining $A \times B$. Define $G_L$ as the group of automorphisms $g : A \times B \to A \times B$ satisfying $\proj_B \circ g = \proj_B$, and $G_R$ those $g$ satisfying $\proj_A \circ g = \proj_A$. The groups obtained, as well as the subgroups $G_L$ and $G_R$, do not depend on the choice of the product $A \times B$ (up to isomorphism of the triples $(G, G_L, G_R)$).

In categories of sets, it is convenient to think of $A \times B$ as indices into a (possibly infinite) matrix, with $A$ indexing the rows and $B$ the columns. This is the suggested convention for mental pictures and is the one used in the proofs. Then $G_L$ is the \emph{coLumn group} that performs a permutation in each column separately (independently of each other) and $G_R$ is the \emph{Row group} that performs independent permutations on rows. Then $G = G_R G_L G_R$ in $\FinSet$ states precisely that any $m$-by-$n$ matrix containing each element of $\{0,1,...,m-1\} \times \{0,1,...,n-1\}$ exactly once can be turned into the matrix $X_{ab} = (a,b)$ by first permuting each row, then each column, then each row.

In group-theoretic terms, since $G_L$ and $G_R$ are subgroups, we can state the weaker fact $G = G_LG_RG_L \cup G_RG_LG_R$ equivalently as follows: the group $G$ is generated by $G_L \cup G_R$, and its diameter with respect to the generating set $G_L \cup G_R$ is at most three (independently of $A$ and $B$). This diameter is in general what we call the \emph{alternation diameter} of $G$, or of the object $A \times B$ having $G$ as automorphism group. The alternation diameter of a category is the least upper bound of alternation diameters of products.

This diameter can infinite for a single object $A \times B$, when $G$ is generated by $G_L$ and $G_R$, but is not equal to a finite number of alternations (this is what happens in \Top\ and \Pos). In principle a category may also have infinite alternation diameter due to objects having arbitrarily large finite alternation diameters. If $G$ is not generated by $G_L$ and $G_R$ at all, we say the alternation diameter is \emph{undefined}, and a category has undefined alternation diameter if some product in it does.


\section{Existing related work}
\label{sec:Existing}

The case of finite sets, which started this paper, is inspired by \cite{Se18}, where it is shown that any permutation of $A \times C \times B$, for three sets $A, C, B$ with $|A| = |B| = 2, |C| \geq 3$, can be written as a composition of finitely many permutations where alternately only $A \times C$ or $C \times B$ is permuted. Our notion of alternation diameter in the category of finite sets is related to this definition, as it also means refers to ``alternately permuting a product on the left and right''. The difference is that there is no communication coordinate (making it harder), but we allow the permutation to depend on the value on the right when permuting the value on the left, and vice versa (making it easier). 

It turns out that the results about finite sets and vector spaces have been proved before in the context of memoryless computation: \cite[Theorem~3]{BuGiTh14} and \cite[Theorem 2]{GaRi15} are essentially the same result as Theorem~\ref{thm:General}. More related results on permutation groups can be found in \cite{CaFaGa14}. The motivation and framework is ostensibly different, but the case of finite sets is proved in \cite[Theorem~3.1]{BuGiTh09} using a version of Hall's theorem.

Theorem~\ref{thm:FinVect} is also known previously: the number $2n-1$ of alternations needed for a product of length $n$ (which can be obtained from Theorem~\ref{thm:FinVect}) can be found in \cite{BuGiTh09,BuGiTh14} (for a larger class of modules). It turns out that $2n-1$ is not optimal at least for finite fields: in \cite[Theorem~2.1]{CaFaGa14a} it is proved that for \FinVect\ over a finite field the optimal number of alternations for a product of length $n$ is $\lfloor 3n/2 \rfloor$.

It should be possible to extract, from the results of \cite{AhBuCi09}, a natural category where alternation diameter is defined but infinite for some objects.

We do not know of previous work on alternation diameter in categories that are less obviously computationally relevant, in particular the results on countably infinite sets and homeomorphisms groups are new to the best of our knowledge.

In Section~\ref{sec:Graphs} we discuss a known related result in graph theory.

\section{Bounded alternation diameter}
\label{sec:FiniteSets}

\subsection{Finite sets}

In this section we look at symmetric groups $\Sym(A \times B)$ for finite sets $A, B$, i.e. automorphism groups of product objects in the category $\FinSet$ of finite sets and functions. Permutations act from the left and compose right-to-left, $(g \circ h)(a) = g(h(a))$ and we also write $g \circ h = gh$.

The following theorem is a well-known corollary of Hall's marriage theorem (see Lemma~\ref{lem:MatrixHallInfinite} for a proof):

\begin{lemma}
\label{lem:MatrixHall}
Let $M$ be an $m$-by-$m$ matrix over $\N$ where all rows and columns sum to $n \geq 1$. Then $M \geq P$ cellwise, for some permutation matrix $P$.
\end{lemma}

This naturally implies that every such matrix is a sum of $n$ permutation matrices, but we prefer to use the lemma directly.

\begin{definition}
Let $A$ and $B$ be finite sets, and let $G = \Sym(A \times B)$. Define the \emph{coLumn group} of $G$ as
\[ G_L = \{ g \in G \;|\; \forall (a,b) \in A \times B: \exists c \in B: g(a,b) = (c,b) \} \]
and define the \emph{Row group} $G_R$ symmetrically.
\end{definition}

\begin{theorem}
\label{thm:Two}
Let $A$ and $B$ be finite sets, and let $G = \Sym(A \times B)$. Then $G = G_R G_L G_R = G_L G_R G_L$.
\end{theorem}

\begin{proof}
We prove $G = G_R G_L G_R$. The equality $G_R G_L G_R = G_L G_R G_L$ follows by symmetry.

A permutation $\pi$ of $A \times B$ can be seen as an $A \times B$-matrix $X$ containing each entry of $A \times B$ exactly once. (Formula: $X_{ab} = \pi(a,b)$.) We show that $\pi \in G_RG_LG_R$ by showing that $\pi G_RG_LG_R$ contains the identity map. In terms of the matrix $X$, precomposing $\pi$ by $(\pi')^{-1} \in G$, i.e. $\pi \mapsto \pi \circ (\pi')^{-1}$, corresponds to permuting the entries of the matrix $X$ by $\pi'$ (in the obvious ``forward'' direction). Thus, our task is to turn $X$ into the matrix $X_{ab} = (a, b)$ by first permuting the rows, then the columns, then the rows again.

Now, ignoring the $B$-component of every matrix entry, and supposing without loss of generality that $A = \{0,1,\ldots,m-1\}$ and $B = \{0,1,\ldots,n-1\}$, we obtain an $m$-by-$n$-matrix $M$ over $A$ where every element of $A$ occurs exactly $n$ times.

To such a matrix $M$, we associate an $m$-by-$m$-matrix $N$ defined as follows:
\[ N_{a,b} = |\{ j \;|\; M_{a,j} = b \}|. \]
Then every row of $N$ sums to $n$ because $M$ is an $m$-by-$n$ matrix, and every column of $N$ sums to $n$ because every $b \in A$ appears $n$ times in $M$. It follows from Lemma~\ref{lem:MatrixHall} that $N \geq P$ where $P$ is a permutation matrix.

We want to permute $M$ so that every column contains every symbol of $A$ exactly once. To do this, consider a row $a \in A$ of $M$, and let $a' \in A$ be the unique element such that $P_{a,a'} = 1$. Then row $a$ of $M$ contains at least one copy of $a'$. On each row, move such an $a'$ to the first column. Since $P$ is a permutation matrix, the elements $a'$ moved to the first column are distinct, so the net effect of this is that the first column of $M$ contains each element of $A$ exactly once. Considering now the $m$-by-$(n-1)$ matrix obtained from $M$ by deleting the first column, we observe that every element of $A$ appears exactly $n-1$ times. By induction, we obtain that $M$ can be permuted by an element of $G_R$ into a matrix where all columns contain each element of $A$ exactly once.

Now, apply an element of $G_L$ to sort each column so that the $a$th row of $M$ contains only $a$s.

Now, consider the action of this transformation on the original $A$-by-$B$ matrix $X$ with entries in $A \times B$. After the transformation (by an element of $G_L G_R$), the row $a$ contains only values of the form $(a, b)$ where $b \in B$. Since every entry $(a,b)$ appears exactly once in $X$, the set of values on row $a$ is then precisely $\{ (a, b) \;|\; b \in B \}$. We can now apply a final permutation in $G_R$ to permute all elements into their correct position, obtaining the matrix representing the identity permutation on $A \times B$.
\end{proof}

Let us make some additional observations. Consider an arbitrary product $X = A_1 \times A_2 \times \cdots \times A_k$. Write $G_i$ for the group of permutations of $X$ that only modify the $i$th component of their input. By induction on $k$, Theorem~\ref{thm:Two} shows the following:

\begin{theorem}
\label{thm:General}
Let $X = A_1 \times A_2 \times \cdots \times A_k$ where $A_i$ are finite sets and let $G_i \leq \Sym(X)$ be as above. Then 
\[ \Sym(X) = G_k G_{k-1} \cdots G_2 G_1 G_2 \cdots G_{k-1} G_k. \]
If each $A_i$ has at least two elements, then no sequence of less than $2k-1$ groups $G_i$ suffices.
\end{theorem}

\begin{proof}
We first prove the formula for $\Sym(X)$. The case $k = 1$ is trivial, and $k = 2$ is Theorem~\ref{thm:Two}. Now, let $A = A_1 \times \cdots \times A_{k-1}$ and $B = A_k$. Then any permutation $\pi$ of $X$ is in $G_R G_L G_R$ where $G_R$ and $G_L$ are defined with respect to the decomposition $X = A \times B$. Let $\pi = \pi_3 \circ \pi_2 \circ \pi_1$ be the corresponding decomposition. Then $\pi_1$ and $\pi_3$ are in $G_k$ because they do not modify the $A$-component of their input.

We can write $\pi_2$ as
$ \pi_2 = \prod_{b \in B} \pi'_b $ 
where each $\pi_b'$ is a permutation of $A$ that modifies the $A$-component only if the $B$-component is equal to $b$. Each permutation $\pi'_b$ is in $G_{k-1}' \cdots G_2' G_1' G_2' \cdots G_{k-1}'$ by induction, where $G'_i$ are the groups corresponding to components of the product $A$. Thus
\[ \pi_2 = \prod_{b \in B} \pi'_b \in (G_{k-1} \cdots G_2 G_1 G_2 \cdots G_{k-1})^\ell, \]
and we can reorder the product to get $\pi_2 \in G_k G_{k-1} \cdots G_2 G_1 G_2 \cdots G_{k-1} G_k$ since the permutations corresponding to distinct $\pi'_b$ commute (as they have disjoint supports).

For the second claim, we show a stronger result: We cannot have
\[ \Sym(X) = G_{i_\ell} \cdots G_{i_2} \cdot G_{i_1} \]
for any sequence where there are at least two indices that occur at most once. Suppose the contrary, and let $i_j = a$, $i_k = b$ with $j < k$, and $a$ and $b$ each occur only once. Suppose that $A_i = \{0,1,...,|A_i|-1\}$.

For $m,n \in \{0,1\}$ write $B_{mn}$ for the set of elements of $X$ where the $A_a$-coordinate is equal to $m$ and the $A_b$-coordinate is equal to $n$. We claim that if $\pi$ is in $G_{i_\ell} \cdots G_{i_2} \cdot G_{i_1}$ and fixes $B_{00}$, then it maps no elements of $B_{10}$ into $B_{01}$, which clearly proves the claim. To see this, observe that all of the groups $G_{i_h}$ except $G_{i_j}$ and $G_{i_k}$ leave the sets $B_{mn}$ invariant. Thus, when $G_a$ gets its turn, elements of $B_{00}$ and $B_{10}$ have not yet been moved. Since $\pi$ fixes $B_{00}$, the $G_a$-permutation cannot move any elements away from $B_{00}$ (since after this step, their $A_a$-coordinate will no longer change). But then elements of $B_{10}$ cannot be moved into $B_{00}$ by $G_a$ since $|B_{00}| = |B_{10}|$, so after applying $G_a$, elements of $B_{10}$ still have nonzero $A_a$-coordinate, which will no longer change. Thus $\pi$ cannot move them into $B_{01}$.
\end{proof}

\begin{corollary}
The category $\FinSet$ has alternation diameter $3$.
\end{corollary}


One can extract a full characterization of $G_{RL}$ from the proof of Theorem~\ref{thm:Two}.

\begin{lemma}
\label{lem:RLCharacterization}
Let $A$ and $B$ be sets, and let $G = \Sym(A \times B)$.
The set $G_R G_L$ contains precisely those permutations $\pi$ such that for $\proj_A : A \times B \to A$ the natural projection map, for all $b \in B$, the map $a \mapsto \proj_A(\pi(a, b)) : A \to A$ is bijective.
\end{lemma}

\begin{proof}
Note that the set $G_R G_L$ contains precisely the permutations that can be mapped to the identity by precomposing first with an element of $G_L$, and then an element of $G_R$. In matrix form, they are the $A$-by-$B$ matrices $X$ over $A \times B$ that can be turned into the matrix $X_{ab} = (a, b)$ by first applying a permutation of columns (from $G_L$) and then a permutation of rows (from $G_R$).\footnote{To readers experiencing chiral confusion, we write some formulas: if we ``first'' apply $G_L$ and ``then'' $G_R$, to a matrix $X$, then formulaically we obtain $G_R \cdot (G_L \cdot X) = (G_L \cdot X) \circ G_R^{-1} = X \circ G_L^{-1} \circ G_R^{-1} = X \circ (G_R G_L)^{-1}$, which contains the identity matrix if and only if $X \in G_R G_L$.}

For sufficiency, observe that the above property is precisely the one that holds after the first application of $G_R$ in the proof of Theorem~\ref{thm:Two} -- in terms of the matrix $M$, it states that every column of $M$ contains every symbol of $A$ exactly once. The two following steps in that proof perform any permutation of $G_R G_L$ using only this property, and do not use the finiteness of $A$ or $B$.

For necessity, write $\pi$ in matrix form, $X_{ab} = \pi(a, b)$, and suppose that for some $\pi' \in G_L, \pi'' \in G_R$, we can turn $X_{ab}$ into the matrix $X_{ab} = (a, b)$ by first applying $\pi'$ and then $\pi''$. First, $a \mapsto \proj_A(\pi(a, b)) : A \to A$ must be injective: otherwise, some column $b$ of $X$ contains both $(a, c)$ and $(a, c')$ for some $c \neq c' \in B$. This still holds after applying $\pi'$, so at the time of the final application of $\pi''$, $(a, c)$ and $(a, c')$ are on the same column, thus cannot both be on the correct row $a$. Second, $a \mapsto \proj_A(\pi(a, b)) : A \to A$ must be surjective: if there exists $c \in A$ and $b \in B$ such that the $b$th column of $X$ does not contain any $(c, b')$, $b' \in B$, then the same is true after applying $\pi'$. In particular the $b$th column will necessarily contain some value $(c', b')$, $c' \neq c$, on the $c$th row, after the application of $\pi'$.
\end{proof}

One can give precise formulas for the sizes of each of the sets obtained by applying $G_L, G_R$ in various orders.

\begin{definition}
\label{def:Gw}
When $G = \Aut(A \times B)$ is an automorphism group, for $w \in \{L, R\}^*$, write $G_w = G_{w_1} G_{w_2} \cdots G_{w_{|w|}}$.
\end{definition}

\begin{theorem}
Let $A, B$ be finite sets, $|A| = m, |B| = n$. For any $w$, $G_w$ is equal to one of the sets in $\{1, G_L, G_R, G_{LR}, G_{RL}, G_{LRL}\}$ and
\[ |G_{LRL}| = mn! \]
\[ |G_{LR}| = |G_{RL}| = m!^n n!^m \]
\end{theorem}

\begin{proof}
Since $G = G_{LRL}$ by Theorem~\ref{thm:Two} and $G_L = G_LG_L$ and $G_R = G_RG_R$ because $G_L$ and $G_R$ are groups, the claim about $G_w$ is true. The first formula comes from the fact that $G_{LRL} = \Sym(A \times B)$. By definition, $|G_{L} \cap G_{R}| = 1$, so since $G_L$ and $G_R$ are subgroups,
\[ |G_{LR}| = |G_L G_R| = \frac{|G_L||G_R|}{|G_L \cap G_R|} = m!^n n!^m, \]
where $|G_L| = m!^n$ since we choose an independent permutation of each of the $n$ columns of size $m$. The formula $|G_{LR}| = |G_{RL}|$ follows by symmetry.
\end{proof}

The above theorem shows that at least three alternations are needed also in the stronger sense that $G \neq G_LG_R \cup G_RG_L$ for $|A|, |B|$ large enough. This is because $|G_{RL} \cup G_{RL}| \leq 2m!n!$ is dwarfed by $mn!$ by a straightforward application of Stirling's formula.

We note that Lemma~\ref{lem:RLCharacterization} gives a characterization of permutations in $G_{LR} \cap G_{RL}$. We have not investigated whether there is a simple formula for the cardinality $|G_{LR} \cap G_{RL}|$ in terms of $m$ and $n$.

From Theorem~\ref{thm:General} we immediately obtain an alternation diameter result for finite-support permutations on infinite sets. Write $\Sym_0(X)$ for the group of finite-support permutations of a set $X$. Define $G_i \leq \Sym_0(X)$ as before, requiring that only the $i$th coordinates of inputs are modified by elements of $G_i$.

\begin{corollary}
Let $X = A_1 \times A_2 \times \cdots \times A_k$ where $A_i$ are arbitrary sets and let $G_i \leq \Sym_0(X)$ be as above. Then 
\[ \Sym_0(X) = G_k G_{k-1} \cdots G_2 G_1 G_2 \cdots G_{k-1} G_k. \]
If each $A_i$ has at least two elements, then no sequence of less than $2k-1$ groups $G_i$ suffices.
\end{corollary}

\begin{proof}
Every permutation has finite support, thus finite projection of the support on the sets $A_i$. Pick suitable finite initial segments of the sets, and for sufficiency apply Theorem~\ref{thm:General}, and for necessity its proof.
\end{proof}

\subsection{Countably infinite sets}
\label{sec:Countable}

We now look at permutations with countably infinite support.

We recall a version of Hall's theorem for infinite sets and include a proof sketch (see e.g. \cite{Di05,CeCo10} for details).

Here, \emph{graphs} are undirected without self-loops, but may have multiple edges between two vertices. A graph is \emph{bipartite} if its vertices can be partitioned into two nonempty sets $L, R$ in such a way that no edge goes between two vertices in $L$ or between two vertices in $R$. Write $N(A)$ for the open neighborhood of a subset $A$ of a graph, i.e. the set of vertices connected to a vertex in $A$ by an edge. A graph is \emph{locally finite} if every vertex has finite degree, i.e.\ $|N(\{a\})| < \infty$ for all vertices $a$. Write $A \Subset B$ for $A$ a finite subset of $B$. A \emph{matching} in a bipartite graph, with a fixed bipartition $L, R$ of the vertices, is a $1$-to-$1$ correspondence that matches a subset of the elements of $L$ injectively to a subset of $R$. We write matchings as partial functions $\alpha : {\subset L} \to R$ 

\begin{lemma}
\label{lem:HallInfinite}
Let $G$ be a locally finite bipartite graph where for each finite set of vertices $A$, $|N(A)| \geq |A|$. Then $G$ admits a perfect matching.
\end{lemma}

\begin{proof}
Let us call the vertices ``left'' or ``right'' depending on which side they are on, $L$ and $R$ as sets. The set of subsets of the edge set has a natural compact topology, namely the product topology on $\{0,1\}^E$ where $E$ is the set of edges and $1$ means the edge is included in the set. Matchings form a closed subset of this space. Since the graph is locally finite, for each vertex $v$ the set of matchings where $v$ is matched is clopen, thus compact. Thus, there exists a matching $\alpha : {\subset L} \to R$ where a maximal set of left vertices is (injectively) matched with some vertex on the right, and for this maximal set, a maximal set of right vertices is matched.

If some left vertex $v$ is not matched in $\alpha$, let $C$ be the set of those vertices (left or right) and edges which are reachable by a path starting from $v$ where every second edge is part of the maximal matching $\alpha$. If $C$ is infinite, by K\"onig's lemma we can find an infinite path, and swapping the edges that are in $\alpha$ with those not in $\alpha$ on this path, we add $v$ to $C$ but remove no vertex from $C$, so $C$ was not maximal. If $C$ is finite, and some $u \in C \cap R$ is not matched, then we can take a path from $v$ to $u$ and again swap the matching edges with non-matching edges to add $v, u$ to $C$ without removing any matched left vertices. If $C$ is finite and all elements of $C \cap R$ are matched, then $|C \cap L| > |C \cap R|$ (because $\alpha$ gives an injection from $C \cap R$ into $(C \cap L) \setminus \{v\}$), a contradiction since all edges from $C \cap L$ are included in $C$ and thus $|C \cap R| = |N(C \cap L)| \geq |C \cap L|$.

We conclude that $t\alpha : L \to R$, i.e. $\alpha$ matches every left vertex with a right vertex. Suppose then that $\alpha$ is not surjective. Then perform the above argument with the roles of left and right reversed, and observe that we also never unmatched a matched right vertex when modifying our matching. Alternatively, one can construct a left-surjective and right-surjective matching separately and apply the Cantor-–Schr\"oder–-Bernstein argument.
\end{proof}

In the following, we use the matrix terminology, though indexing by infinite sets. The meaning should be clear.

\begin{lemma}
\label{lem:MatrixHallInfinite}
Let $A$ be any set and $N$ any $A$-by-$A$ matrix over $\N$. If every row and column sums to $n \geq 1$ (in particular, cofinitely many entries have value $0$ on every row and column), then $N \geq P$ for some permutation matrix $P$.
\end{lemma}

\begin{proof}
Construct the bipartite graph $G$ with a copy of $A$ ``on the left'' and another copy of $A$ ``on the right''. Include an edge from left-$a$ to right-$b$ if $N_{a, b} > 0$. Then this graph satisfies the assumptions of the previous lemma: Clearly it is locally finite and bipartite. Consider any finite set of left vertices $L' \Subset L$. We have
\begin{align*}
|L'| &= \frac{1}{n} \sum_{a \in L', b \in N(L')} N_{a, b} \\
&\leq \frac{1}{n} \sum_{b \in N(L')} \sum_{a \in L} N_{a, b} = |N(L')|.
\end{align*}
Similarly, for any $R' \Subset R$ we have $|N(R')| \geq |R'|$. The previous lemma gives a perfect matching, i.e. a permutation matrix $P \leq N$.
\end{proof}

\begin{theorem}
\label{thm:CountableCase}
Let $A,B$ be sets. If $|A|, |B| \leq \aleph_0$, then
\[ \Sym(A \times B) = G_{LRLR} \cup G_{RLRL}. \]
If both $A$ and $B$ are infinite, then $\Sym(A \times B) \neq G_{LRLR}$ and $\Sym(A \times B) \neq G_{RLRL}$. If $B$ is infinite and $|A| \geq 2$, then $\Sym(A \times B) \neq G_{RLR}$.
If $B$ is finite, then $\Sym(A \times B) = G_{RLR}$.
\end{theorem}

\begin{proof}
In the first claim, if $A$ and $B$ are finite, then this is Theorem~\ref{thm:Two}, and if only one of them is finite, this follows from the last claim (which we prove last). We thus consider the case that both are infinite.

As usual, we consider the $A$-by-$B$ matrix $X$ over $A \times B$ and the cellwise projection $M$ with values in~$A$, representing a permutation $\pi$ of $A \times B$. We will compose $\pi$ with elements of $G_L, G_R, G_L, G_R$ in that order. Again we use the standard left action, which formulaically is precomposition with inverse, and in terms of matrices corresponds to directly permuting the entries.

We prove the following:
\begin{itemize}
\item if any cofinite set of columns of $M$ contains infinitely many distinct values of $A$, i.e.
\[ \forall A' \Subset A, B' \Subset B: \exists a' \notin A', a \in A, b \notin B': M_{ab} = a', \]
then $\Sym(A \times B) = G_R G_L G_R G_L$,
\item if the dual claim holds, then $\Sym(A \times B) = G_L G_R G_L G_R$,
\item either the claim or its dual holds.
\end{itemize}

We first prove the third item. Suppose that $M$ does not satisfy the first item. Then a finite set of columns of $M$ contains all values from a cofinite subset $C$ of $A$. Then those columns of $X$ must contain all pairs $(a,b)$ with $a \in C$, so $b \in B$ takes every possible value in these finitely many columns. This clearly cannot happen in finitely many rows and finitely many columns. Thus the dual claim of the first item holds.

Thus we only need to prove the first item, as the second is symmetric, and by the third claim, these together give $\Sym(A \times B) = G_{LRLR} \cup G_{RLRL}$.

Assume then that any cofinite set of columns of $M$ contains infinitely many distinct values of $A$. We describe what happens to the matrices $M$ and $X$ after each step (we do not rename them after each step). Our plan is the following:
\begin{enumerate}
\item After applying $G_{L}$, every row of $M$ has infinitely many distinct values, and every $a \in A$ that appears on infinitely many columns also appears on infinitely many rows.
\item After an application of $G_R$, all columns of $M$ have exactly one copy of each $a \in A$.
\item By Lemma~\ref{lem:RLCharacterization}, another application of $G_{RL}$ finishes the proof.
\end{enumerate}

Assume $A = \N, B = \N$. In the first step, we modify each column at most once, and then freeze it and never modify it again. We go through $n \in \N$, and on the $n$th turn, we modify a finite set of columns and then freeze them. What we ensure on the $n$th turn is that the first $n$ rows of $M$ all contain at least $n$ distinct values that appear frozen columns, and that each $a \in A$ with $a < n$ which appears on infinitely many columns also appears on at least $n$ distinct rows on frozen columns.

No trick is needed, we just do it: After a finite number of steps, we have seen only frozen finitely many columns, and $M$ contains infinitely many values in any cofinite set of columns, so we never run out of fresh values $a \in A$ to move to rows needing them, thus we can indeed make sure each row contains more and more distinct values, and if $a \in A$ appears on infinitely many columns, then we can make this choice infinitely many times.

In the limit, in the compact topology of cellwise convergence of the matrix entries, clearly the resulting matrix still describes a permutation (we performed a permutation at most once on each column, so clearly the transformation is columnwise bijective, thus bijective). Every row contains infinitely many distinct values since for any row $a$ and any $n \geq a$, on the $n$th turn we made sure the $a$th row contains at least $n$ values.


In the second step, we again construct the permutation column by column, but now we are permuting rows instead of columns. We modify each row infinitely many times, but with smaller and smaller supports, and take the limit of the process. Note that the set of matrices representing injective maps $A \times B \to A \times B$ is closed with respect to this topology, so the limit is automatically injective (if well-defined). Matrices representing surjective maps are not closed, so our matrix may fail to be surjective in the limit if we are careless. To ensure surjectivity, we fix an enumeration $i_0, i_1, i_2, ...$ of $A \times B$, and say the \emph{index} of $(a,b)$ is the $n$ such that $i_n = (a,b)$. For surjectivity, it is enough that each $(a,b)$ appears in the limit, and for this it is enough that from some point on, we no longer move pairs $(a,b)$ with indices up to $n$.

To ensure that we get a limit in the second step, we modify each column at most once. We go through $n \in \N$, and on the $n$th turn, we permute the rows in such a way that the $j$th column is not modified for $j < n$, the $n$th column contains exactly one copy of each $a \in A$ after the permutations, and the $n$th pair is moved to column $n$ unless it is already in some column $j < n$.

Call/color an element of the $n$th column \emph{red} if the row containing it has not yet been permuted on the $n$th turn, and green otherwise. We begin the turn by moving the $n$th pair to column $n$ if it is among columns $j > n$ (and color it green, so that at this point we have at most one green entry). We now perform a back-and-forth argument where we modify each row at most once by alternating the following steps:
\begin{itemize}
\item Pick the next $a \in A$ that has not yet been moved into the column $n$ (i.e. all occurrences of $a$ in it are red), move that $a$ into the column (possibly it was already there), color it green and freeze the row.
\item Let $k \in A$ be the first row with a red symbol and move any fresh $a \in A$ (that does not yet appear as a green symbol in the $n$th column) into it.
\end{itemize}

The first type of move in the back-and-forth is always possible: If $a \in A$ appeared on only finitely many columns initially (thus also after the first step), in which case it appeared on infinitely many distinct rows, and since exactly $n$ copies of $a$ are on frozen columns $j < n$, there are still infinitely many unfrozen copies of $a$ on infinitely many distinct rows. If $a \in A$ appeared on infinitely many columns before the first step, then after the first step $a \in A$ appears on infinitely many rows, thus on the $n$th turn there are again still unfrozen copies of $a$ on infinitely many distinct rows.

The second type of move in the back-and-forth is always possible: All rows contain infinitely many distinct symbols $a \in A$, and at any point of the process we have only finitely many green symbols on the column.

This concludes the construction, as Lemma~\ref{lem:RLCharacterization} applies to the resulting matrix.

For the second claim, we observe that no $A \times B$-matrix with $A$-projection
\[ \left(\begin{matrix}
* & 0 & 0 & 0 & 0 & \cdots \\
* & 0 & 0 & 0 & 0 & \cdots \\
* & 0 & 0 & 0 & 0 & \cdots \\
* & 0 & 0 & 0 & 0 & \cdots \\
* & 0 & 0 & 0 & 0 & \cdots \\
\vdots & \vdots &\vdots &\vdots &\vdots & \ddots 
\end{matrix}\right) \]
where the $*$-symbols are elements of $A \setminus \{0\}$,
is in $G_R G_L G_R G_L$. Namely, the first application of $G_L$ is useless, as the set of matrices of this form is invariant under $G_L$. The following application of $G_R$ will move at most one nonzero $a \in A$ into each row. Then already the top left $3$-by-$3$ block necessarily contains at least two zeroes on some column, so Lemma~\ref{lem:RLCharacterization} does not apply. Note that matrices of this form do exist since $|B| = |A \times B|$.

For the third claim, let $\{0, 1\} \subset A$, $0 \in B$. Pick a bijection $\phi : B \mapsto B \setminus \{0\}$. Consider a permutation $\pi$ mapping $(0, 0) \mapsto (1, 0)$, $(0, \phi(b)) \mapsto (0, b)$, $(1, b) \mapsto (1, \phi(b))$. Writing $\pi$ again as an $A$-by-$B$ matrix $X$, the $0$-row contains exactly one element which should be in the $1$-row in the end, and the $1$-row contains only elements that belong to it. We have $\pi \notin G_R G_L G_R$: After the application of $G_R$ to $\pi$, we still have exactly one element of the form $(1,b)$ (namely $(1,0)$) in the $0$-row. If it is in the $b$-column, then the $b$-column contains two elements $(1,m), (1,n)$ of this form, and thus after the application of $G_L$, the $b$-column contains such an element in some row $a \neq 1$. Therefore after applying $G_R$ we still have at least one element of the form $(1,n)$ in the $a$-row, so we have not turned $\pi$ to the identity.

For the fourth claim, the proof is that of Theorem~\ref{thm:Two}, but using Lemma~\ref{lem:MatrixHallInfinite} in place of Lemma~\ref{lem:MatrixHall}.
\end{proof}

\begin{corollary}
The category $\CountSet$ has alternation diameter $4$.
\end{corollary}

Since in $\FinSet$ we had alternation diameter $3$ and $G = G_{LRL}, G = G_{RLR}$, while in $\CountSet$ we have alternation diameter $4$ and $G = G_{LRLR} \cup G_{RLRL}$ but \emph{not} $G = G_{LRLR}$ or $G = G_{RLRL}$, one notes that ``alternation diameter'' indeed loses some information. One could instead mimic quantifier hierarchies and define $\Sigma_1 = G_L$, $\Pi_1 = G_R$, and inductively $\Sigma_{i+1} = G_L \Pi_i$, $\Pi_{i+1} = G_L \Sigma_i$, $\Delta_i = \Sigma_i \cap \Pi_i$. Then in $\FinSet$, the alternation hierarchy collapses on the level $\Delta_3$, while the one for $\CountSet$ collapses at the join $\Sigma_4 \cup \Pi_4$.

\subsection{Finite-dimensional vector spaces}

Besides $\Set$, an obvious place to look for category-wide diameter bounds is the category of finite-dimensional vector spaces. The first reason is that it is a category where all objects and morphisms behave nicely. The second is the intuition that dimension can often replace cardinality. We find that alternating diameter is indeed $3$, as in $\FinSet$.

Fix a field $k$ and let $\FinVect_k$ be the category of finite-dimensional vector spaces over $k$.

\begin{theorem}
\label{thm:FinVect}
The category $\FinVect_k$ has alternation diameter $3$. More precisely, let $A$ and $B$ be in $\FinVect_k$. Then $G = \Aut(A \times B) = GL(A \times B)$ satisfies $G = G_L G_R G_L = G_R G_L G_R$.
\end{theorem}

\begin{proof}
Let $m = \dim(A), n = \dim(B)$. The claims $G = G_L G_R G_L$ and $G = G_R G_L G_R$, are symmetric, so we only prove the first.

Consider an arbitrary matrix $M \in \GL(m + n, k)$ in block representation with four blocks, of widths and heights $m$ and $n$, respectively (the ``$A$-by-$A$ block'' of size $m$-by-$m$ on the top left). Now applying automorphisms in $G_L$ (from the left) amounts to row operations that do not modify the bottom blocks, and $G_R$ modifies the bottom blocks only. We need to turn $M$ into the identity matrix with an element of $G_L G_R G_L$.

If $v_1, \ldots, v_{m+n}$ are the rows of $M$, then their restriction to their first $m$ coordinates has full rank. It is easy to see that then there is a finite sequence of row operations that do not affect the bottom rows -- that is, an element of $G_L$ -- left multiplication by which turns the top left block into the identity matrix.

Next, apply an element of $G_R$ to turn the bottom left block into zeroes, and then the bottom right block into the identity matrix. Finally, apply an element of $G_L$ to turn the top left block into the all-zero matrix. We have shown that $M$ can be turned into the identity matrix by multiplying it by an element of $G_{LRL}$, thus $M \in G_{LRL}$.

To see that this is optimal, it is enough to consider $k = A = B$ and show that $G = GL(A \times B)$ does not satisfy $G = G_L G_R$.

Algebraically, it is easy to see that $\left( \begin{smallmatrix} 0 & 1 \\ 1 & 0 \end{smallmatrix} \right) \notin G_L G_R$. Namely, after applying a row operation that modifies only the second row (element of $G_R$) to the identity matrix, the first row is still $(1 \; 0)$ so the second row cannot be $(1 \; 0)$, and thus an application of $G_L$ cannot turn the resulting matrix into $\left( \begin{smallmatrix} 0 & 1 \\ 1 & 0 \end{smallmatrix} \right)$.
\end{proof}

In the case $k = A = B = \R$, one can also verify $\left( \begin{smallmatrix} 0 & 1 \\ 1 & 0 \end{smallmatrix} \right) \notin G_L G_R$ geometrically by staring intently at a square.

Special linear groups $SL(A \times B)$ also have alternation diameter $3$ in an obvious sense, by the same proof as for $GL(A \times B)$. (This fact does not directly fit the framework of this paper, in that the author does not know whether $SL(A \times B)$ can be seen (in a natural way) as the automorphism group of a product object in a category.)

\subsection{Graphs}
\label{sec:Graphs}

We mention a related result from graph theory. The box product $G \square H$, sometimes called the Cartesian product of $G$ and $H$ defined by
\begin{align*}
((g,h), (g',h')) \in E(G \square H) \iff &((g,g') \in E(G) \wedge h = h') \vee \\
& (g = g' \wedge (h,h') \in E(H)),
\end{align*}
(though it is not the category-theoretic product in the usual category of simple graphs) admits unique prime decompositions for finite connected graphs, and automorphisms of $G \square H$ are essentially entirely determined by $G$ and $H$ (and a bit of counting) in the sense that if we decompose $G$ and $H$ into their prime factors, every automorphism consists of a permutation of the factors (with respect to a fixed identification of isomorphic factors), followed by separate permutations of the factors. In this sense, connected graphs with respect to box product have ``bounded alternation diameter up to reordering of prime factors''. See \cite{ImKlRa08} for details. For some related observations see Remark~\ref{rem:CPP} and Section~\ref{sec:Q}.

\section{Left and right groups}
\label{sec:GeneralDef}

In this section we perform the (rather trivial) diagram chasing and algebra required to show that $G_L$ and $G_R$ ``make sense'', i.e.\ are actually subgroups, and are independent of the choice of the product object $A \times B$.

We give the diagrammatic definition of these subgroups. For a product object $A \times B$ with defining projections $\pi_A : A \times B \to A$ and $\pi_B : A \times B \to B$, write $G = \Aut(A \times B)$ and write $G_L$ for the set of elements $f \in \Aut(A \times B)$ such that the leftmost diagram below commutes in Figure~\ref{fig:DiagDef} (resp. $G_R$ for the set of elements $g$ such that the rightmost diagram commutes).
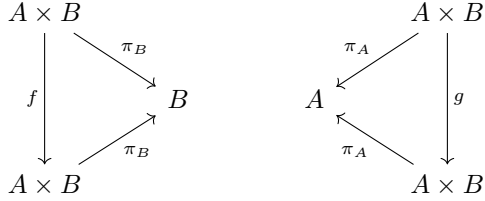
\begin{figure}[h!]
\begin{center}
\begin{tikzcd}
A \times B \arrow{rd}{\pi_B} \arrow[swap]{dd}{f} &  \\
& B \\
A \times B \arrow[swap]{ur}{\pi_B} &
\end{tikzcd} \hspace{1cm}
\begin{tikzcd}
& A \times B \arrow[swap]{ld}{\pi_A} \arrow{dd}{g} \\
A & \\
& A \times B \arrow{lu}{\pi_A}  
\end{tikzcd}
\end{center}
\caption{Definition of $G_L$ (on the left) and $G_R$ (on the right).}
\label{fig:DiagDef}
\end{figure}

It is easy to see (by gluing diagrams, or by algebra) that $G_L$ and $G_R$ are submonoids of $\Aut(A \times B)$ under composition. To see that they are groups, note that if $f \in G_L$ and $fg = gf = \id_{A \times B}$ then $\pi_B f = \pi_B \implies \pi_B f g = \pi_B g \implies \pi_B = \pi_B g \implies g \in G_L$. From this, it follows $G_L$ (symmetrically $G_R$) is indeed a subgroup of $G$.

For an object $B$ in a category $\mathcal{C}$, the \emph{over category above $B$}, denoted $\mathcal{C}/B$, is the category whose objects are morphisms $f : C \to B$ in $\mathcal{C}$ (or simply the morphisms themselves), and morphisms from $f : C \to B$ to $g : D \to B$ are morphisms $h : C \to D$ in $\mathcal{C}$ such that the following diagram commutes:
\begin{center}
\begin{tikzcd}
C \arrow{rr}{h} \arrow{dr}{f} & & D \arrow{dl}{g} \\
& B & 
\end{tikzcd} 
\end{center}

Applying some geometric transformations to this diagram reveals a similarity with $G_L$ and $G_R$, and we can make the observation that the above proof that $G_L$ is a group actually shows that $G_L$ is the automorphism group of the morphism $\pi_B : A \times B \to B$ as an object of the over category above $B$. Similarly $G_R$ is the automorphism group of $\pi_A$ in the over category above $A$.

To see that the choice of the product object does not matter, suppose $C$ is another product of $A$ and $B$, and $\pi_A' : C \to A, \pi_B' : C \to B$ the defining projections. By the universal property, there is a unique isomorphism $\phi : C \to A \times B$ such that the following diagram commutes:

\begin{center}
\begin{tikzcd}
& C \arrow[swap]{ld}{\pi_A'} \arrow{rd}{\pi_B'} \arrow{dd}{\phi} &  \\
A & & B \\
& A \times B \arrow{lu}{\pi_A} \arrow[swap]{ru}{\pi_B} & 
\end{tikzcd} 
\end{center}

Then $\phi$ is an isomorphism between $\pi_B$ and $\pi_B'$ and thus gives an isomorphism of their automorphism groups in the over category $\mathcal{C}/B$, which as discussed are the groups $G_L$ (corresponding to the two different choices of the product object). The same applies to $G_R$. 
 
We summarize the discussion into a theorem. Define a \emph{group triple} to be a triple of groups $(G_1, G_2, G_3)$ such that $G_2, G_3 \leq G_1$. We say two group triples $(G_1, G_2, G_3)$ and $(H_1, H_2, H_3)$ are isomorphic is an isomorphism $\phi : G_1 \to H_1$ such that $\phi(G_2) = H_2, \phi(G_3) = H_3$.

\begin{theorem}
Let $\mathcal{C}$ be a category and let $A \times B$ be a product of objects $A$ and $B$ with defining projections $\pi_A : A \times B \to A$ and $\pi_B : A \times B \to B$. Then the sets $G_L$ and $G_R$ defined by the diagrams in Figure~\ref{fig:DiagDef} are subgroups of $G = \Aut(A \times B)$ under composition. The resulting triple $(G, G_L, G_R)$ does not depend on the choice of the product $A \times B$, up to isomorphism of group triples.
\end{theorem}

\section{Undefined alternation diameter}
\label{sec:Undefined}

\subsection{Topological spaces}

It seems that, not surprisingly, in typical categories with a lot of structure, the alternation diameter is undefined for the whole category, that is, left and right automorphisms do not generate all others. We give in $\Top$ a non-trivial example (the plane) of undefined alternation diameter, and also a trivial example (the square $[0,1]^2$), and in $\Pos$ we exhibit an object with has trivial left and right groups, but non-trivial (though not far from trivial) automorphism group.

In the category $\Top$ of topological spaces and continuous functions, the automorphism group $\Homeo([0,1] \times [0,1])$ has undefined alternation diameter, since the left border and the top border can be exchanged by an automorphism, but this obviously cannot be done by $G_L$ or $G_R$. We state this as a metalemma.\footnote{We add ``meta'' to distinguish this from the lemma which would be obtained by replacing ``nice'' by the best possible list of necessary properties, which can be deduced from the proof.}

\begin{metalemma}
\label{mlem:Flip}
Let $\mathcal{C}$ be a nice enough concrete category, and $X \times X$ a product object. If $X \times X$ has a definable subset of the form $Y \times X \cup X \times Y \neq X \times X$, then $G = \Aut(X \times X)$ is not equal to $G_w$ for any $w \in \{L, R\}^*$.
\end{metalemma}

\begin{proof}
We have $G_R(Y \times X) = Y \times X$ by the definition of $G_R$, and $G_R(X \times Y) \subset Y \times X \cup X \times Y$ since $Y \times X \cup X \times Y$ is definable. Since $G_R$ is a group action and $G_R(Y \times X) = Y \times X$, we must have $G_R(X \times Y) = X \times Y$. Similarly, $G_L(X \times Y) = X \times Y$ and $G_L(Y \times X) = Y \times X$. The flip $f(x, y) = (y, x)$ is in $\Aut(X \times X)$ by the universal property of $X \times X$ (see Lemma~\ref{lem:Flip} for a diagrammatic deduction). Since it does not (setwise) stabilize $Y \times X$, the subgroups $G_L, G_R$ do not generate it, thus they do not generate $\Aut(X \times X)$.
\end{proof}

\begin{corollary}
The category $\Top$ has undefined alternation diameter.
\end{corollary}

For homogenous spaces the question is more interesting. Let us show that also $\Homeo((0,1) \times (0,1)) = \Homeo(\R \times \R)$ has undefined alternation diameter, by showing that rowwise and columnwise homeomorphisms cannot untangle sufficiently wild homeomorphisms in finite time.

\begin{theorem}
\label{thm:Homeo}
The automorphism group $G = \Homeo(\R \times \R)$ has an element that is not in $G_w$ for any $w \in \{L, R\}^*$.
\end{theorem}

\begin{proof}
To agree with our matrix convention ($G_R$ permutes the Rows), draw the axes of the plane so that the second is the horizontal axis (left-to-right) and the first axis is vertical (top-down).

Consider a homeomorphism $\alpha : \R \times \R \to \R \times \R$. Let $p$ be the unit speed path $p(x) = (x, 0)$ from $(0,-1)$ to $(0,0)$, and consider the $\alpha$-image $\alpha \circ p$ of this path, which is a path from $\alpha(0,-1)$ to $\alpha(0,0)$. Then $g \circ \alpha \circ p$ is a path from $g(\alpha(0,-1))$ to $g(\alpha(0,0))$ for any homeomorphism $g$.

Now, cut out a small compact neighborhood $U$ of $\alpha(0, 0)$ and consider the sequence in which $\alpha \circ p(x)$ traverses the rays $E, N, W, S$ in cardinal directions eminating from $\alpha(0,0)$, before it first enters $U$ (ignoring repeated crossings of the same ray). Let $u \in \{E, N, W, S\}^*$, be the (finite) word thus obtained. By the intermediate value theorem, between occurrences of $N$ and $S$ there is an occurrence of $E$ or $W$.\footnote{The word $u$ may depend on the choice of $U$, and there need not be a best possible choice for which this word is the longest. What is important is that some choice gives a long word. Formally, one can consider the set of all words that correspond to some choice of $U$ to obtain a more canonical invariant.}

Now consider $g \circ \alpha$ for some $g \in G_R$, and consider the corresponding word $w'$, computed up to the neighborhood $g(U)$. Observe that $g$ changes either the orientation of all rows, or none of the rows. Then in $w'$, we have at least as many alternations between $E$ and $W$ as in $w$. Similarly, an application of $G_L$ cannot decrease the number of alternations between $S$ and $N$.

It follows that if the path $p$ from $(0,-1)$ to $(0,0)$ is mapped by $\alpha$ to a path having $(ENWS)^{k+1}$ as a subsequence of the word $u$ corresponding to some choice of a small neighborhood of $\alpha(0,0)$, the corresponding spiral in each homeomorphism in $(G_R \circ G_L)^k \alpha$ has a corresponding word with subsequence $ENWS$, thus is not the identity map, as the identity map preserves $p$, and $p$ does not spiral with respect to any choice of $U$.

Of course, the points $(0,-1), (0,0)$ are not in any way special. By including such spirals of all finite diameters in our homeomorphism $\alpha$ by twisting horizontal paths from $(a, b)$ to $(a, b')$, we obtain a homeomorphism with undefined alternation diameter. One can also have infinitely many twists around the same point: through the usual identification $\R^2 \cong \C$ the homeomorphism $\alpha$ defined by
\[ \alpha(r e^{2 \pi i t}) = r e^{2 \pi i (t + \frac{1}{r})} \]
is not in $G_w$ for any $w \in \{L, R\}^*$, as any $(ENWS)^{k+1}$ appears as a subword of $u$ corresponding to a small enough choice of $U$ around the origin. See Figure~\ref{fig:Spiral} for a visualization of this homeomorphism.
\end{proof}

\begin{figure}
\begin{center}
\raisebox{-0.45\height}{\includegraphics[scale=0.4]{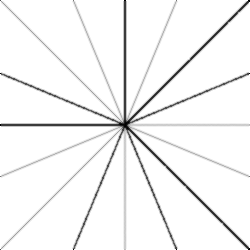}}
$\;\;\;\;\;\; \mapsto \;\;\;\;\;\;$
\raisebox{-0.45\height}{\includegraphics[scale=0.4]{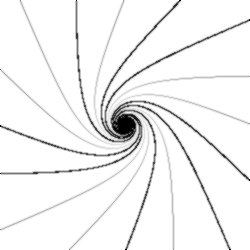}}
\end{center}
\caption{A twisted self-homeomorphism of the plane that cannot be untangled with finitely many alternations of $G_L$ and $G_R$, shown here twisting finitely many lines (of pseudorandom shades) eminating from the origin.}
\label{fig:Spiral}
\end{figure}

By a more careful analysis, one can construct homeomorphisms with word norm $n \in \N$ from the identity map in the homeomorphism group of $\R^2$ w.r.t. the generators $G_L \cup G_R$. It follows that for $G = \Aut(\R \times \R)$, $\langle G_L \cup G_R \rangle$ is not equal to $G_w$ for any $w \in \{L, R\}^*$. We do not have a global understanding of this group $\langle G_L \cup G_R \rangle$.

\subsection{Posets}
\label{sec:Posets}

In many categories, there are even easier ways to find undefined alternation diameter than Metalemma~\ref{mlem:Flip}. We now prove that under rather general assumptions, the flip automorphism used from Metalemma~\ref{mlem:Flip} is not generated by $G_L$ and $G_R$. We show how to apply them to some (finite) examples in \Pos.

\begin{lemma}
\label{lem:Flip}
Let $\mathcal{C}$ be a category and $G = \Aut(X \times X)$. If $G_L$ is trivial and the first and second canonical projections $\pi_i : X \times X \to X$ are distinct, then $G$ has undefined alternation diameter.
\end{lemma}

\begin{proof}
Let $\pi_1 : X \times X \to X$ be the defining left projection and $\pi_2 : X \times X \to X$ be the defining right projection. Define the flip automorphism $f : X \times X \to X \times X$ as follows: let $Y = X \times X$ and define a left and right projection, respectively $\pi_1', \pi_2'$, by $\pi_1' = \pi_2, \pi_2' = \pi_1$. The universal property yields a morphism $f : Y \to X \times X$ satisfying $\pi_1 \circ f = \pi_2$, $\pi_1 \circ f = \pi_2$. By the assumption, $\pi_1 \neq \pi_2$, so $f \neq \id_{X \times X}$.

From the existence of $f$ we see that if $G_L$ is trivial, also $G_R$ is trivial, as $f$ conjugates $G_R$ onto $G_L$. Thus, $\langle G_L \cup G_R \rangle$ is trivial. But $G$ is non-trivial, as $f \neq \id$.
\end{proof}

If $(A, \leq_A)$ and $(B, \leq_B)$ are partially ordered sets, a \emph{weakly increasing function} is $f : A \to B$ such that $a \leq_A b \implies f(a) \leq_B f(b)$. Let $C_n$ denote the \emph{$n$-chain} $\{0,1,...,n-1\}$ under the usual ordering.

\begin{proposition}
In the category \Pos\ of posets and weakly increasing functions, $C_n \times C_n$ has undefined alternation diameter.
\end{proposition}

\begin{proof}
A meta-addition to the above lemma, which we use in this proof, is that if a category is such that $G_L$ for $G = \Aut(A \times B)$ always consists of automorphisms of $A$ applied separately in each fiber, and $\Aut(A)$ is trivial, then so is $G_L$ for $G = \Aut(A \times A)$, and thus $A \times A$ has undefined alternation diameter by the previous lemma. We do not attempt to give general conditions that guarantee this, but obviously $\Pos$\ has this property. (Note that category-theoretic products are simply Cartesian products with componentwise comparison, so all fibers carry an isomorphic induced order.)

Let $C_n \times C_n$ be the category-theoretic product of $C_n$ with itself in $\Pos$. We have trivial $G_L$ for $G = \Aut(C_n \times C_n)$ since $\Aut(C_n)$ is trivial by the observation in the previous paragraph, so the previous lemma shows that $G$ has undefined alternation diameter.
\end{proof}

One can also prove that the ``boundary'' of $C_n \times C_n$ (one of the elements has maximal or minimal value) is definable, and apply Metalemma~\ref{mlem:Flip}.

\begin{corollary}
The category \Pos\ has undefined alternation diameter.
\end{corollary}

We give another general reason why posets have undefined alternation diameter. For $G = \Aut(X \times Y)$, in nice enough concrete categories define the \emph{pure left group} as the group $G_L'$ of those $h \in \Aut(X \times Y)$ satisfying that for some $g \in \Aut(X)$, $h(x, y) = (g(x), y)$ for $x \in X, y \in Y$. In a general category, diagrammatically, we define the pure left group as the group of solutions $h$ to the rightmost diagram in Figure~\ref{fig:ComDiag} for various choices of $g \in \Aut(X)$. Similarly, define the \emph{pure right group} $G_R'$. One can show (by purely diagrammatic reasoning) that $G_L'$ and $G_R'$ are commuting subgroups of $G$.

The dual of the over category in Section~\ref{sec:GeneralDef} is the under category. If $\mathcal{C}$ is a category and $A \in \mathcal{C}$ is an object, the objects of the \emph{under category below $A$}, denoted $A/\mathcal{C}$, are morphisms $\phi : A \to B$ in $\mathcal{C}$, and morphisms between $\phi, \psi \in A/\mathcal{C}$ are morphisms of $\mathcal{C}$ between codomains of $\phi$ and $\psi$ such that the obvious diagram commutes.

\begin{lemma}
Let $\mathcal{C}$ be a category and $G = \Aut(X \times X)$. If $G_L = G_L'$, $G_R = G_R'$, and $\Aut(X)$ is nontrivial, then $G$ has undefined alternation diameter.
\end{lemma}

\begin{proof}
We first show $f$, defined as in Lemma~\ref{lem:Flip}, is not equal to $\id_{X \times X}$ under the assumption that $\Aut(X)$ is nontrivial. Observe first that the projections $\pi_i : X \times X \to X$ are always epic by studying the leftmost of the two commutative diagrams in Figure~\ref{fig:ComDiag}.

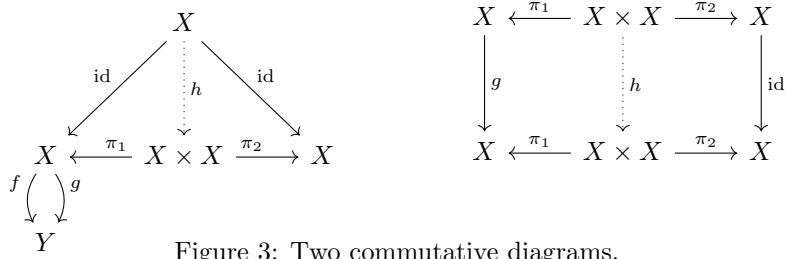
\begin{figure}
\begin{center}
\begin{tikzcd}
  & X \arrow[swap]{ldd}{\id} \arrow[dotted]{dd}{h} \arrow{rdd}{\id} & \\
  & & \\
X \arrow[bend right,swap]{d}{f} \arrow[bend left]{d}{g} & X \times X \arrow[swap,near start]{l}{\pi_1} \arrow[near start]{r}{\pi_2} & X \\
Y & &
\end{tikzcd} \hspace{1.3cm}
\begin{tikzcd}
X \arrow{dd}{g}  & X \times X \arrow[swap]{l}{\pi_1} \arrow{r}{\pi_2} \arrow[dotted]{dd}{h} & X \arrow{dd}{\id} \\
 &            &  \\
X  & X \times X \arrow[swap]{l}{\pi_1} \arrow{r}{\pi_2} & X \\
& & \\
& &
\end{tikzcd}
\end{center}
\vspace{-1cm}
\caption{Two commutative diagrams.}
\label{fig:ComDiag}
\end{figure}


Observe that if $f = \id_{X \times X}$, then $\pi = \pi_1 = \pi_2$, and the rightmost commutative diagram in Figure~\ref{fig:ComDiag} yields $g \circ \pi = \id \circ \pi$, so $g \in \Aut(X) \implies g = \id_X$ by epicness, contradicting the nontriviality of $\Aut(X)$.

The above shows that $\pi_1 \neq \pi_2$. We need the stronger fact that $\pi_1$ and $\pi_2$ are not even isomorphic in the under category below $X \times X$. Suppose they were, and $g \in \Aut(X)$ is such that $g \circ \pi_1 = \pi_2$. Now consider the leftmost diagram in Figure~\ref{fig:ComDiag2}.

\begin{figure}
\begin{center}
\begin{tikzcd}[cramped]
& X \arrow[swap]{ldddd}{\id} \arrow{rdddd}{\id} \arrow[dotted, "{h}" description]{dd} & \\
& & \\
& X \times X \arrow{ldd}{\pi_1} \arrow[swap]{rdd}{\pi_2} \arrow["{\pi_1}" description]{dd} & \\
& & \\
X  & X \arrow{l}{\id} \arrow[swap]{r}{g} & X
\end{tikzcd} \hspace{1.4cm}
\begin{tikzcd}[cramped]
X \arrow{dd}{g_L}  & X \times X \arrow[swap]{l}{\pi_1} \arrow{r}{\pi_2} \arrow[dotted]{dd}{f} & X \arrow{dd}{g_R} \\
 &            &  \\
X  & X \times X \arrow[swap]{l}{\pi_1} \arrow{r}{\pi_2} & X \\
& &
\end{tikzcd}
\end{center}
\caption{Two more commutative diagrams.}
\label{fig:ComDiag2}
\end{figure}
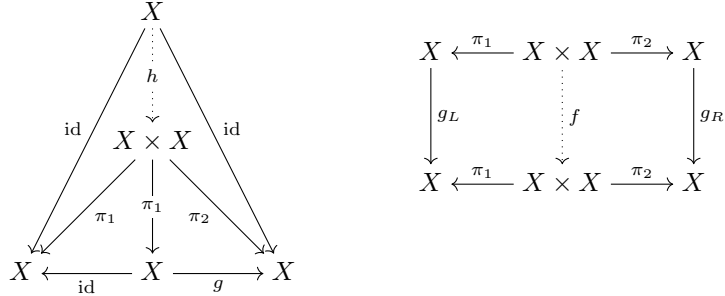

This diagram clearly commutes before and after $h$ is obtained from the universal property of $X \times X$. Two readings show that $g \circ \pi_1 \circ h = \id_X$ and $\pi_1 \circ h = \id_X$, so $g = \id_X$. But this contradicts $\pi_1 \neq \pi_2$, shown above.

We now show that the assumptions $G_L = G_L'$ and $G_R = G_R'$ imply that $f$ is not in $\langle G_L, G_R \rangle$. To see this, observe that $\langle G_L \cup G_R \rangle = \langle G_L' \cup G_R' \rangle = G_L' \times G_R'$ (an internal direct product in $G$) since $G_L'$ and $G_R'$ commute. The diagrammatic statement of $f \in G_L' \times G_R'$ (seen similarly as the commutation of $G_L'$ and $G_R'$) is the fact that the rightmost diagram in Figure~\ref{fig:ComDiag2} commutes, where $f$ is the flip and $g_R$ and $g_L$ are some automorphisms of $X$. Thus $\pi_2 = \pi_1 \circ f = g \circ \pi_1$ for some $g \in \Aut(X)$, showing that $\pi_2$ and $\pi_1$ are isomorphic in the under category below $X \times X$, contradicting what we showed.
\end{proof}

Let now $D$ be the \emph{diamond}, i.e.\ the poset $\{0,1\}^2$ under cellwise comparison. Note that this is a lattice, i.e.\ we can define operations $\vee, \wedge$ mapping a pair of elements $a, b$ to their unique supremum $a \vee b$ and infimum $a \wedge b$.

\begin{proposition}
In the category \Pos\ of posets and weakly increasing functions, $D \times D$ has undefined alternation diameter.
\end{proposition}

\begin{proof}
The automorphism group of $D \times D$ consists of order-automorphisms of $\{0,1\}^4$ under cellwise comparison. A short analysis shows that $G_L = G_L'$, $G_R = G_R'$, but clearly $\{0,1\}^2$ has nontrivial automorphism group. The previous lemma shows that the alternation diameter is undefined.
\end{proof}

\begin{remark}
\label{rem:CPP}
A more fancy reason that $G_L = G_L'$ and $G_R = G_R'$ in the above example is that for a universal-algebraic variety such as that of lattices, products are the obvious kind, and the variety of lattices has the \emph{congruence-product property} stating that congruences of products are products of congruences. See \cite[Theorem~13]{Gr71} for a proof; the name was used in \cite{SaTo12a} where we applied this property in the context of cellular automata. Indeed, the congruence-product property implies a similar automorphism group description as Section~\ref{sec:Graphs} for finite lattices, i.e.\ every object is a direct product in a unique way, and automorphisms come from reordering the product and applying automorphisms to the directly irreducible components. In this sense, finite lattices also have ``bounded alternation diameter up to reordering of prime factors''.
\end{remark}

Maximilien Gadouleau has shown (private communication) that for a finite poset $X$, the flip is in $\langle G_L \cup G_R \rangle$ if and only if $X \times X$ has defined alternation depth if and only if $X$ is a trivial poset (with no comparable pairs $x \neq y, x \leq y$).

\section{Questions}
\label{sec:Q}


It seems that assuming the axiom of choice (but not necessarily CH), the proof of Theorem~\ref{thm:CountableCase} generalizes at least to pairs of sets $A, B$ with $|A| = |B|$, and we conjecture that such pairs have alternation diameter $4$. We do not know what the general diameter is in $\Set$.

In the case of vector spaces, we do not know what happens with infinite-dimensional vector spaces (possibly with additional structure), or when $k$ is replaced by a ring.

In the case of topological spaces, we do not know which numbers appear as alternation diameters of products $A \times B$, and we do not know the alternation diameter of $\R^m \times \R^n$ as a function of $m, n$. Another interesting direction would be to study less wild homeomorphisms. We conjecture that on the plane, piecewise linear self-homeomorphisms with finitely many polygonal pieces have infinite (as opposed to undefined) alternation diameter.

We do not know if numbers $3$ and $4$ that appeared for sets and vector spaces have any general significance. Do such upper bounds follow from a more general property? Is there is a natural common generalization of Theorem~\ref{thm:Two} and Theorem~\ref{thm:FinVect}? Does every number $n$ appear as the alternation diameter of a (hopefully reasonably natural) category, or at least a product object? Are there (naturally occurring) categories where the alternation diameter is defined but infinite? (This is possibly answered in \cite{AhBuCi09}.) Are there ones where the diameter is infinite, but all objects have finite alternation diameter?

There are many interesting categories not mentioned here where alternation diameter can be studied. Typically, we expect that category will have undefined alternation diameter, in that $\Aut(X \times Y)$ is simply not generated by $G_L$ and $G_R$. Nevertheless, it is interesting to ask how big diameters can be when they exist, and what the nontrivial situations are where $G_L$ and $G_R$ do generate $\Aut(X \times Y)$. In cases where they do not, one can ask what subgroup of $\Aut(X \times Y)$ they generate, and how much needs to be added to obtain all of $\Aut(X \times Y)$.

A natural generalization of this study is to replace the category-theoretic product by another type of product. One can consider products with abstract properties such as tensor products of vector spaces (cf. \cite{La03}), or ones that are otherwise evidently natural such as the box product of graphs discussed in Section~\ref{sec:Graphs} (there are several graph product notions).

One may also study the dual of the definition of alternation diameter in (natural) concrete categories, or equivalently consider the definition in opposite categories of concrete categories. It seems that in certain settings, something morally stronger happens, and the automorphism group of a product object is fully described by the automorphism groups of the components, though we may need to pick the right product decomposition, similarly as what happened in Section~\ref{sec:Graphs} with graphs under the box product. 

For example, in the opposite category of finite graphs (homomorphisms of graphs map edges to edges, but can also map non-edges to edges -- the opposite category has the inverses of such maps as morphisms), prime decomposition into a product project means decomposition into connected components, and is well-known to be unique up to ordering. Automorphisms preserve the components of this product since paths map to paths in a graph automorphism, and we are lead to the known representation of the automorphism group of a finite graph as (a direct product of) wreath products of symmetric groups acting by permuting isomorphic components, together with automorphism groups of the individual components, i.e. we have described the automorphism group of a product object in terms of the components of the product. It seems that this kind of behavior is quite common, but we do not know any nice general conditions.

As mentioned in the introduction, the idea of alternation diameter arose from the study of the category of subshifts in \cite{SaSc16a}, where we looked at some special situations where the alternation diameter is finite. The category of all ($\Z$-)subshifts and block maps was studied in \cite{SaTo15a} under the name K4, and its products are the obvious ones. We do not know what alternation diameters can appear for products in this category. An easy application of Metalemma~\ref{mlem:Flip} shows that the category itself has undefined alternation diameter: Let $X$ be the $\Z$-sunny-side-up, i.e. the orbit closure of the characteristic function of $\{0\} \subset \Z$. Then $X \times \{0^\Z\} \cup \{0^\Z\} \times X$ is the Cantor-Bendixson derivative of $X^2$, thus definable.

Finally, all these questions can be studied for endomorphism monoids $\End(A \times B)$. Some results in this direction can be found by following the citations in Section~\ref{sec:Existing}.

\section*{Acknowledgements}

We thank Peter Selinger for a brief discussion in Reversible Computation~2016 after his talk, where he shared his intuition that (contrary to what I thought) $A \times B$ should have finite alternation diameter for finite $A$ and $B$. We thank Maximilien Gadouleau for pointing out several existing results (Section~\ref{sec:Existing}) on alternation diameter in the literature and for other helpful remarks.

\bibliographystyle{plain}
\bibliography{../../../bib/bib}{}

\def\ocirc#1{\ifmmode\setbox0=\hbox{$#1$}\dimen0=\ht0 \advance\dimen0
  by1pt\rlap{\hbox to\wd0{\hss\raise\dimen0
  \hbox{\hskip.2em$\scriptscriptstyle\circ$}\hss}}#1\else {\accent"17
  #1}\fi}\def\cprime{$'$}
\begin{thebibliography}{10}

\bibitem{AhBuCi09}
Mumtaz Ahmad, Serge Burckel, and Adam Cichon.
\newblock {\em Sequential decomposition of operations and compilers
  optimization}.
\newblock PhD thesis, INRIA, 2009.

\bibitem{BuGiTh09}
Serge Burckel, Emeric Gioan, and Emmanuel Thom{\'e}.
\newblock Mapping computation with no memory.
\newblock In Cristian~S. Calude, Jos{\'e}~F{\'e}lix Costa, Nachum Dershowitz,
  Elisabete Freire, and Grzegorz Rozenberg, editors, {\em Unconventional
  Computation}, pages 85--97, Berlin, Heidelberg, 2009. Springer Berlin
  Heidelberg.

\bibitem{BuGiTh14}
Serge Burckel, Emeric Gioan, and Emmanuel Thom{\'e}.
\newblock Computation with no memory, and rearrangeable multicast networks.
\newblock {\em Discrete Mathematics \& Theoretical Computer Science}, 16, 2014.

\bibitem{CeCo10}
T.~Ceccherini-Silberstein and M.~Coornaert.
\newblock {\em Cellular Automata and Groups}.
\newblock Springer Monographs in Mathematics. Springer-Verlag Berlin
  Heidelberg, 2010.

\bibitem{Di05}
Reinhard Diestel.
\newblock Graph theory 3rd ed.
\newblock {\em Graduate texts in mathematics}, 173, 2005.

\bibitem{GaRi15}
Maximilien Gadouleau and S{\o}ren Riis.
\newblock Memoryless computation: new results, constructions, and extensions.
\newblock {\em Theoretical Computer Science}, 562:129--145, 2015.

\bibitem{Gr71}
George Gr{\"a}tzer.
\newblock {\em Lattice theory. {F}irst concepts and distributive lattices}.
\newblock W. H. Freeman and Co., San Francisco, Calif., 1971.

\bibitem{ImKlRa08}
Wilfried Imrich, Sandi Klav{\v{z}}ar, and Douglas~F Rall.
\newblock {\em Topics in graph theory: Graphs and their Cartesian product}.
\newblock AK Peters/CRC Press, 2008.

\bibitem{La03}
Yves Lafont.
\newblock Towards an algebraic theory of boolean circuits.
\newblock {\em Journal of Pure and Applied Algebra}, 184:2003, 2003.

\bibitem{Ma71}
Saunders MacLane.
\newblock {\em Categories for the working mathematician}.
\newblock Springer-Verlag, New York, 1971.
\newblock Graduate Texts in Mathematics, Vol. 5.

\bibitem{CaFaGa14a}
Ben~Fairbairn Peter J.~Cameron and Maximilien Gadoleau.
\newblock Computing in matrix groups without memory.
\newblock {\em Chicago Journal of Theoretical Computer Science}, 2014(8),
  November 2014.

\bibitem{CaFaGa14}
Ben~Fairbairn Peter J.~Cameron and Maximilien Gadoleau.
\newblock Computing in permutation groups without memory.
\newblock {\em Chicago Journal of Theoretical Computer Science}, 2014(7),
  November 2014.

\bibitem{SaSc16a}
Ville Salo and Michael Schraudner.
\newblock Automorphism groups of subshifts through group extensions.
\newblock Preprint.

\bibitem{SaTo12a}
Ville {Salo} and Ilkka {T{\"o}rm{\"a}}.
\newblock {On Shift Spaces with Algebraic Structure}.
\newblock {\em ArXiv e-prints}, March 2012.

\bibitem{SaTo15a}
Ville Salo and Ilkka T{\"{o}}rm{\"{a}}.
\newblock Category theory of symbolic dynamics.
\newblock {\em Theor. Comput. Sci.}, 567:21--45, 2015.

\bibitem{Se18}
Peter Selinger.
\newblock A finite alternation result for reversible boolean circuits.
\newblock {\em Science of Computer Programming}, 151:2 -- 17, 2018.
\newblock Special issue of the 8th Conference on Reversible Computation
  (RC2016).

\end{thebibliography}

\end{document}